\documentclass[a4paper,leqno,12pt,twoside]{amsart}
\usepackage{amsmath,amsfonts,amssymb,amsthm,amscd}
\usepackage[utf8]{inputenc}
\usepackage[T1]{fontenc}
\usepackage[english]{babel}
\usepackage[colorlinks=true,citecolor=blue, urlcolor=blue, linkcolor=blue,pagebackref]{hyperref}
\usepackage[top=1.2in,bottom=1.2in,left=1.in,right=1.in]{geometry} 
\usepackage{graphicx}
\usepackage{paralist}
\usepackage{tabto}
\usepackage{standalone}
\usepackage{tikz}
\usetikzlibrary{matrix}
\usetikzlibrary{arrows}
\usepackage{xfrac}
\usepackage{amsthm, amsmath, amssymb,latexsym}
\usepackage{tikz-cd}
\usepackage[all]{xy} 

\usepackage{enumerate}
\usepackage{xcolor}
\usepackage{aliascnt}
\usepackage{cleveref}

\newtheorem{theorem}{Theorem}[section]
\newtheorem{lemma}[theorem]{Lemma}
\newtheorem{corollary}[theorem]{Corollary}
\newtheorem{proposition}[theorem]{Proposition}
\newtheorem{remark}[theorem]{Remark}
\newtheorem{definition}[theorem]{Definition}

\def\ra{\rightarrow}

\newcommand{\Sym}{{\operatorname{Sym}}}

\newcommand{\restr}[1]    {\vert_{#1}}

\title{Higher Gaussian maps for special classes of curves}

\author{D. Faro}

\address{Dario Faro  \\ Universit\`a degli Studi di Pavia  \\ Dipartimento di Matematica \\ Via Ferrata 1  \\ 27100 Pavia, Italy  }
 \email{dario.faro@unipv.it}

\author{P. Frediani}
\address{Paola Frediani  \\ Universit\`a degli Studi di Pavia  \\ Dipartimento di Matematica \\ Via Ferrata 1  \\ 27100 Pavia, Italy  }
 \email{paola.frediani@unipv.it}

 \author{A. Lacopo}
 \address{Antonio Lacopo \\ Universit\`a degli Studi di Pavia  \\ Dipartimento di Matematica \\ Via Ferrata 1  \\ 27100 Pavia, Italy  }
 \email{antonio.lacopo01@universitadipavia.it}

\begin{document}

\begin{abstract}
In this paper we study higher Gaussian (or Wahl) maps for the canonical bundle of certain smooth projective curves. More precisely, we determine the rank of  higher Gaussian maps of the canonical bundle for plane curves, for curves contained in certain linear systems in a surface given by a product of two curves and for curves contained in a sufficiently ample line bundle on an Enriques surface. 
\end{abstract}

\thanks
{D. Faro, P. Frediani and A. Lacopo are members of GNSAGA (INdAM) and are partially supported by PRIN project {\em Moduli spaces and special varieties} (2022).}

\maketitle

\section{Introduction}
Let $C$ be a smooth projective curve. The $k$-th Gaussian map (or Wahl map) is the map  
$$
 \gamma^k_C: H^0(C\times C,I^k_{\Delta_C}\otimes K_C^{\boxtimes 2})\rightarrow H^0(C,K_C^{\otimes (k+2)}),
$$
induced by  the short exact sequence $ 0 \rightarrow I^{k+1}_{\Delta_C} \rightarrow
I^k_{\Delta_C} \rightarrow K^{\otimes k}_C \rightarrow 0$.  Wahl showed in \cite{wahl1} that if $C$ sits on a $K3$ surface, then  $\gamma^1_C$ is not surjective (see also \cite{bm}). On the other hand, Ciliberto, Harris and Miranda proved that for the general curve of genus $g \geq 10$, $g \neq 11$, $\gamma^1_C$ is surjective (see \cite{chm}, \cite{voi} for another proof). Actually also the converse holds: a Brill-Noether Petri general curve with non surjective (first) Wahl map lies on a  $K3$ surface or on a limit thereof, as shown in \cite{abs}. More generally,  the (co)rank of the first Wahl map is related to the extendability of the curve as a linear section of a higher dimensional variety, as shown in \cite{cds}. 

The main purpose of this work is to determine the (co)rank of higher Wahl maps for certain classes of smooth projective curves. The study of  higher Gaussian maps has been carried out mainly in two directions. On the one hand, as in the case $k=1$, to give obstructions for the extendability of curves. This is the case of the main result of \cite{cfp}, where it is shown  that if a curve lies on an abelian surface, then the second Wahl map is not surjective (whereas the first is surjective if the curve belongs to a sufficiently positive linear system). This again has to be  compared with the main result in (\cite{ccm}), which states the surjectivity of the second Wahl map for the general curve of genus gretater than or equal to $18$.

On the other hand, the study of higher Wahl maps has been applied to the study of the local geometry of the Torelli locus inside the moduli space of principally polarized abelian varieties. More precisely, in \cite{fredianihigher} and \cite{ffl} it is shown that there is a relation between the second fundamental form of the Torelli morphism and higher even Wahl maps. This generalized the main result in \cite{cpt} concerning the second Wahl map.

The first main result of this work  is the computation of the (co)rank of higher Wahl maps for plane curves.
\begin{theorem}
\label{theorema} (see Theorem \ref{mainmain})
If $C$ is a smooth plane curve of degree $d \geq 6$, then for all  $0 \leq k \leq \frac{d-6}{2}$, we have 
$$rank(\gamma^k_C) = \frac{2k+3}{2}(d(d-3)-k(k+3)) = \frac{2k+3}{2}(2g-2 -k(k+3)),$$ 
$$corank(\gamma^k_C) = k(k+3)\frac{2k+3}{2}.$$
\end{theorem}
We observe that the corank of $\gamma^k_C$ is independent of the degree of the plane curve, as previously shown  by Wahl in the case $k=1$ (\cite[Remark 4.9]{wahl}). 

The second main result gives a sufficient condition for the surjectivity of higher Gaussian maps  for curves lying in the product of two curves. This generalizes results of Wahl and Colombo - Frediani for cases $k=1$ (\cite[Theorem 4.11]{wahl}), and $k=2$ respectively (\cite[Theorem 3.1]{cf-Michigan}).
\begin{theorem}
\label{propsurj} (see Theorem \ref{surjffsrusi})
Let  $k \geq 2$ be an integer. Let $C_i$, $i=1,2$ be smooth projective algebraic curves of genus $g_i$, let $D_i$ be an effective divisor on $C_i$ of degree $d_i$. Denote by $p_i:X:=C_1 \times C_2 \ra C_i$  the projections and suppose that
\begin{itemize}
\item[1.] $g_1 \geq 2, g_2 \geq 1 $ or $g_1 \geq 1, g_2 \geq 2 $,  and $d_i \geq kg_i + k+3$ for $i=1,2$ or,

\item[2.] $g_1=0$, $g_2 \geq 2$, $d_1 > 2(k+1)$, $d_1 > \frac{kd_2}{g_2-1}$, $d_2 \geq kg_2 +k+3 $.

\end{itemize}
 Then for any irreducible smooth curve $C$ in the linear system  $ |p_1^*(D_1) \otimes p_2^*(D_2)|$,  $\gamma^k_{C} $ is surjective. 
\end{theorem}
This in turn gives another proof of the surjectivity of higher Gaussian maps for the general curve of genus $g$, where $g$ can take an infinite number of values. The surjectivity of higher Gaussian maps for the general curve of genus big enough was already shown  by Rios Ortiz in \cite{ro}. 

\begin{corollary}
Let $k \geq 2$. For all $g_i$ and $d_i$ satisfying the hypothesis of  Theorem \ref{propsurj},  the general curve of genus  
\begin{equation}
g=1+(g_2-1)d_1+(g_1-1)d_2+d_1d_2,
\end{equation}
has surjective $k$th Gaussian map.
\end{corollary}

The last part of the paper deals with the study of canonical Gaussian maps for curves on Enriques surfaces, and it is based on previous results by Faro and Spelta for Prym-canonical Gaussian maps \cite{farospelta}. 

The main result is the following 

\begin{theorem}
\label{teoenriques} (see Theorem \ref{enriques})
Let $X$ be an unnodal Enriques surface and $k  \in \mathbb{N}$. If $C$ is a smooth curve with $\phi(\mathcal{O}_X(C)) > 4(k+2)  $, then $\gamma^k_C$ is surjective. If $k=1$ it is sufficient to require. $\phi(\mathcal{O}_X(C)) > 6 $.
\end{theorem}

Here the function $\phi$ is a measure of the positivity of a line bundle on an Enriques surface (see section \ref{section5} for the definition).

The structure of the paper is as follows. In section \ref{section2} we recall the definition of the Gaussian maps and some useful variations. In section  \ref{section3} we consider the situation of plane curves, proving  Theorem \ref{theorema}. In section  \ref{section4} we prove Theorem \ref{propsurj}. In  section  \ref{section5} we give the proof of  Theorem \ref{teoenriques}.

\section{Gaussian maps}
\label{section2}
Let $Y$ be a smooth complex projective variety, let $\Delta_Y\subset
Y\times Y$ be the diagonal and let $p,q: Y \times Y \rightarrow Y$ be the two projections. Let $L$ and $M$ be line bundles on $Y$.
For a non-negative integer $k$, the \emph{k-th Gaussian map}
associated with these data is the restriction map to the diagonal
\begin{equation}\label{gaussian1}H^0(Y\times Y,I^k_{\Delta_Y}\otimes p^*L\otimes q^*M )  \stackrel{\gamma^k_{L,M}}\longrightarrow
H^0(Y,{I^k_{\Delta_Y}}_{|\Delta_Y}\otimes L\otimes M)\cong
H^0(Y,S^k\Omega_Y^1\otimes L\otimes M).
\end{equation}
The exact sequence
\begin{equation}
\label{Ik} 0 \rightarrow I^{k+1}_{\Delta_Y} \rightarrow
I^k_{\Delta_Y} \rightarrow S^k\Omega^1_Y \rightarrow 0,
\end{equation}(where $S^k\Omega^1_Y$ is identified to its image via the diagonal map), tensored by $p^*L \otimes q^*M$, shows that the domain of the $k$-th
Gaussian map is the kernel of the previous one:
$$\gamma^k_{L,M}:
ker \gamma^{k-1}_{L,M}\rightarrow H^0(S^k\Omega_Y^1\otimes L\otimes
M).$$ In our applications, we will  also assume that the two line bundles
$L$ and $M$ coincide.

Assume $Z$ is a smooth irreducible divisor on $Y$. 
We consider a variant of Gaussian maps that  appear in \cite{voi}, \cite{cfp}. They are defined
 as follows: let $L$ be a line bundle on $Y$, and let $M_Z$ be a
 line bundle on $Z$, seen as a sheaf on $Y$. The $k$-th order Gaussian map
 associated with these data is
\begin{equation}\label{gaussian2}
 \gamma_{L,M_Z}^k:H^0(Y\times Y,I^k_{\Delta_Y}\otimes p^*L\otimes
 q^* M_Z)\rightarrow H^0(Y,S^k{\Omega^1_Y}\otimes L\otimes M_Z).
 \end{equation}
 Sequence (\ref{Ik}), tensored with
 $p^* L\otimes q^* M_Z$ remains exact. In fact, since $q^*M_Z$ is locally free on $Y \times Z$, and $I^k_{\Delta_Y}/I^{k+1}_{\Delta_Y}$ is locally free on $\Delta_Y$, it suffices to show that $tor_1^{{\mathcal O}_{Y \times Y}}({\mathcal O}_{\Delta_Y}, {\mathcal O}_{Y \times Z})=0.$
Consider the following exact sequence   \  \
 $0 \rightarrow I_{Y \times Z} \rightarrow
  {\mathcal O}_{Y \times Y} \rightarrow {\mathcal O}_{Y \times Z} \rightarrow 0,$ \
 and tensor it by ${\mathcal O}_{\Delta_Y}$. 
 We obtain $$0 \rightarrow    tor_1^{{\mathcal O}_{Y \times Y}}({\mathcal O}_{\Delta_Y}, {\mathcal O}_{Y \times Z})  \rightarrow  I_{Y \times Z} \otimes  {\mathcal O}_{\Delta_Y} \rightarrow
  {\mathcal O}_{\Delta_Y} \rightarrow {\mathcal O}_{\Delta_Z} \rightarrow 0,$$
  where $\Delta_Z$ is the diagonal of $Z$ in $Y \times Z$. Since   $I_{Y \times Z} \otimes  {\mathcal O}_{\Delta_Y}$ is the ideal sheaf of $\Delta_Z$ in $\Delta_Y$, we get $tor_1^{{\mathcal O}_{Y \times Y}}({\mathcal O}_{\Delta_Y}, {\mathcal O}_{Y \times Z})=0$.
Hence, as above, the domain of
 these $k$-th Gaussian maps is the kernel of the previous ones:
\begin{equation}\gamma^k_{L,M_Z}: ker \gamma^{k-1}_{L,M_Z}\rightarrow
H^0(S^k{\Omega^1_Y}\otimes L\otimes M_Z).
\end{equation}

Let ${\mathcal F}$ be a coherent sheaf on $Y$ which is locally free on $Z$ and apply  $p_*$ to the exact sequence \eqref{Ik} tensored by $q^*{\mathcal F}$. Setting $$R^k_{Y, {\mathcal F}}:=  p_* (I^k_{\Delta_Y} \otimes q^*{\mathcal F}),$$ we get 
\begin{equation}
\label{R}
0 \rightarrow R^{k+1}_{Y, {\mathcal F}} \rightarrow R^k_{Y, {\mathcal F}} \rightarrow S^k \Omega^1_Y \otimes {\mathcal F} \rightarrow R^1p_*(I^{k+1}_{\Delta_Y} \otimes q^*{\mathcal F}) \rightarrow R^1p_*(I^{k}_{\Delta_Y} \otimes q^*{\mathcal F})
\end{equation}
If ${\mathcal F } = L$ is a line bundle on $Y$, $R^k_{Y,L}$ are vector bundles on $Y$. Tensoring \eqref{R} by another line bundle $M$ and taking global sections we obtain the Gaussian maps defined in \eqref{gaussian1} as the maps 

\begin{equation}
\label{gaussR}
\gamma^k_{L,M}: H^0(Y, R^k_{Y,L} \otimes M) \rightarrow H^0(S^k \Omega^1_Y \otimes L \otimes M).
\end{equation}
When $L=M$ we will set $\gamma_L^k:= \gamma^k_{L,L}$. 

In the paper we will consider the following setup: let $X$ be a smooth projective surface, $H$ an ample line bundle on it, $C \in |H|$ a smooth projective curve, so by adjunction $K_C = (K_X \otimes H)_{\restr C}$. Set $L:= K_X \otimes H$. 

We have the following commutative diagram: 
\[\begin{tikzcd}
\label{general}
	H^0(X, R^k_{X,L} \otimes L)  && H^0(X,S^k\Omega^1_X(2L)) \\
	&&& H^0(C,S^k\Omega^1_{X}(2L)_{\mid_C}) \\
	H^0(C, R^k_{C, K_C} \otimes K_C) && H^0(C,K_C^{\otimes (k+2)}).
	\arrow["\gamma^{k}_{L}",from=1-1, to=1-3]
	\arrow[from=1-1, to=3-1]
	\arrow["p_1",from=1-3, to=2-4]
	\arrow["p_2",from=2-4, to=3-3]
	\arrow["\gamma^k_C",from=3-1, to=3-3]
\end{tikzcd}\]

where the left vertical map and the map $p_1$  are the restriction maps. 
The map $p_2$ is constructed as follows: let us consider the conormal bundle exact sequence
$$
\begin{tikzcd}
    0 \arrow[r] & \mathcal{O}_C(-C) \arrow[r] & \Omega^1_{X_{\mid C}} \arrow[r] & K_C \arrow[r] & 0,
\end{tikzcd}
$$
take its $k$-th symmetric power 

\begin{equation}
\label{p2}
    0 \longrightarrow  S^{k-1}\Omega^1_{{X}_{\restr C}}(-C)  \longrightarrow S^k \Omega^1_{X_{\mid C}}  \longrightarrow K_C^{\otimes k} \longrightarrow 0,
\end{equation}
 tensor it  by $K_C^{\otimes 2}$ and take global sections. 

So to study $\gamma_C^k$ we will study the maps $\gamma^k_L$, $p_1$ and $p_2$.

\section{Gaussian maps on plane curves}
\label{section3}

In this section we will deal with the following setup:
$X= {\mathbb P}^2$, $C$ will be a smooth plane curve  of degree $d$, $L = {\mathcal O}_{{{\mathbb P}^2}}(d-3)$, so that $L_{|C}= K_C$. 
We set

\begin{equation}
\label{R_Y}
R^k_{{\mathbb P}^2}:= R^k_{{\mathbb P}^2,L}=  p_* (I^k_{\Delta_{{\mathbb P}^2}} \otimes q^*({\mathcal O}_{{\mathbb P}^2}(d-3))),
\end{equation}

\begin{equation}
\label{G}
G^k:= R^k_{{\mathbb P}^2,K_C}=  p_* (I^k_{\Delta_{{\mathbb P}^2}} \otimes q^*(K_C)),
\end{equation}

\begin{equation}
\label{R_C}
R^k_C:= R^k_{C,K_C}=  p_* (I^k_{\Delta_C} \otimes q^*K_C).
\end{equation}

We denote by 
\begin{equation}
\label{gammak}
\gamma^k_C:= \gamma^k_{K_C, K_C}: H^0(C, R^k_C \otimes K_C) \rightarrow H^0(C, K_C^{\otimes (k+2)}),
\end{equation}
the $k$-th Gaussian map of the canonical bundle of the curve $C$, 
and by 
\begin{equation}
\label{gammakL}
\gamma^k_L:= \gamma^k_{L, L}: H^0({\mathbb P}^2, R^k_{{\mathbb P}^2} \otimes L) \rightarrow H^0(\mathbb{P}^2,S^k\Omega^1_{\mathbb{P}^2}(2L)),
\end{equation}
the $k$-th Gaussian map of the line bundle $L$ on ${\mathbb P}^2$. 
We want to compute the rank of $\gamma^k_C$, the $k$-th Gaussian map of the canonical bundle of the curve.



The restriction map $p_1$  in  diagram  \eqref{general} comes from the following exact sequence 

\begin{equation}
\label{projection1}
    \begin{tikzcd}
    0 \arrow[r] & S^k\Omega^1_{\mathbb{P}^2}(d-6) \arrow[r] & S^k \Omega^1_{\mathbb{P}^2}(2d-6) \arrow[r] & S^k \Omega^1_{\mathbb{P}^2_{\mid C}}(2d-6) \arrow[r] & 0.
\end{tikzcd}
\end{equation}

\begin{lemma}
Let $k \leq d-3$. Then  $\gamma^{k}_{L}$ is surjective.
\end{lemma}
\begin{proof}
This follows from  \cite[ Corollary 2.1]{ro}.  
\end{proof}
Now we only have to study $p_1$ and $p_2$. First, we will show the surjectivity of $p_1$.
\begin{proposition}
\label{surjp1}
    With our assumptions, if $k\leq \frac{d-6}{2}$, then $p_1$ is surjective.
\end{proposition}
\begin{proof}
    Notice that by \eqref{projection1}, it is enough to prove that $H^1(S^k\Omega^1_{\mathbb{P}^2}(d-6))=0$. From \cite[Theorem 1.3.1]{et}, if $k\leq \frac{d-6}{2}$, $H^i(S^k\Omega^1_{\mathbb{P}^2}(d-6))=0$ for $i>0$.
\end{proof}
Let us now focus on $p_2$.
\begin{proposition}
\label{1estimate}
    Assume that $k\leq \frac{d-4}{2}$. Then $$corank(p_2)=k(k+3)\frac{2k+3}{2},$$ hence 
    $$rank(p_2) = \frac{2k+3}{2}(d(d-3)-k(k+3)).$$
    So if $k\leq \frac{d-6}{2}$, we have 
    $$rank(\gamma_C^k)  \geq rank(p_2) = \frac{2k+3}{2}(d(d-3)-k(k+3)).$$ 
    
    \end{proposition}
\begin{proof}
We need to study  exact sequence  \eqref{p2} tensored by $K_C^{\otimes 2}$, namely the  following  
\begin{equation}
    \label{projection2}
    \begin{tikzcd}
        0 \arrow[r] & S^{k-1}\Omega^1_{\mathbb{P}^2_{\mid C}}(d-6) \arrow[r] & S^k \Omega^1_{\mathbb{P}^2_{\mid C}}(2d-6) \arrow[r] & K_C^{\otimes (k+2)} \arrow[r] & 0.
    \end{tikzcd}
\end{equation}
First we  compute $H^1(S^{k-1}\Omega^1_{\mathbb{P}^2_{\mid C}}(d-6)).$\\
In order to do that, we consider the exact sequence 
\begin{equation}
\label{sym}
\begin{tikzcd}
    0 \arrow[r] & S^{k-1}\Omega^1_{\mathbb{P}^2}(-6) \arrow[r] & S^{k-1} \Omega^1_{\mathbb{P}^2}(d-6) \arrow[r] & S^{k-1} \Omega^1_{\mathbb{P}^2_{\mid_ C }}(d-6) \arrow[r] & 0
\end{tikzcd}
\end{equation}
obtained by tensoring with $S^{k-1}\Omega^1_{\mathbb{P}^2}(d-6)$ the restriction sequence 
$$
\begin{tikzcd}
    0 \arrow[r] & \mathcal{O}_{\mathbb{P}^2}(-C) \arrow[r] & \mathcal{O}_{\mathbb{P}^2} \arrow[r] & \mathcal{O}_C \arrow[r] & 0.
\end{tikzcd}
$$
Again, using  \cite[Theorem 1.3.1]{et}, we see that if $k-1\leq \frac{d-6}{2}$, or equivalently $k\leq \frac{d-4}{2}$, we have $H^i(S^{k-1}\Omega^1_{\mathbb{P}^2}(d-6))=0$ for $i>0$. \\ Hence by \eqref{sym} we get $H^1(S^{k-1}\Omega^1_{\mathbb{P}^2_{\mid C}}(d-6))\cong H^2(S^{k-1}\Omega^1_{\mathbb{P}^2}(-6))$. Again by  \cite[Theorem 1.3.1]{et}, we see that $$\dim (H^2(S^{k-1}\Omega^1_{\mathbb{P}^2}(-6)))=\dim(S^{k-1}\mathbb{C}^3\otimes S^{k+2}\mathbb{C}^3)-
\dim(S^{k-2}\mathbb{C}^3\otimes  S^{k+1}\mathbb{C}^3)=
k(k+3)\frac{2k+3}{2}.$$
Now we claim that $H^1(S^k\Omega^1_{\mathbb{P}^2_{\mid C}}(2d-6))=0.$ Indeed, let us consider the following exact sequence 
$$
\begin{tikzcd}
    0 \arrow[r] & S^k\Omega^1_{\mathbb{P}^2}(d-6) \arrow[r] & S^k \Omega^1_{\mathbb{P}^2}(2d-6) \arrow[r] & S^k\Omega^1_{\mathbb{P}^2_{\mid_ C }}(2d-6) \arrow[r] & 0.
\end{tikzcd}
$$
Since we already know that $H^2(S^k\Omega^1_{\mathbb{P}^2}(d-6))=0$, we only have to show that $H^1(S^k\Omega^1_{\mathbb{P}^2}(2d-6))=0$. Indeed, from \cite[Theorem 1.3.1]{et},  this is true if $2d-6-k\geq k$, which is equivalent to $k\leq d-3$. Since we have $k\leq \frac{d-4}{2}\leq d-3$, we are done.\\
Finally, putting all these things together, from \eqref{projection2} we get 
$$
\begin{tikzcd}
     H^0(S^k \Omega^1_{\mathbb{P}^2_{\mid C}}(2d-6))\arrow["p_2",r] & H^0(K_C^{\otimes k+2})\arrow[r] & H^1(S^{k-1}\Omega^1_{\mathbb{P}^2_{\mid C}}(d-6))\cong \mathbb{C}^{k(k+3)\frac{2k+3}{2}} \arrow[r] & 0
\end{tikzcd}
$$
which gives precisely $corank(p_2)=k(k+3)\frac{2k+3}{2}$.
The last statement follows from the commutativity of  diagram \eqref{general} and by the surjectivity of $\gamma^k_L$ and $p_1$. 

\end{proof}


We will show that for $k \leq \frac{d-4}{2}$, the image of $\gamma^k_C$ is contained in the image of $p_2$, hence by Proposition \ref{1estimate} we conclude that if $k\leq \frac{d-6}{2}$ the rank of $\gamma^k_C$ is equal to $rank(p_2) = \frac{2k+3}{2}(d(d-3)-k(k+3)).$
We have the following 

\begin{lemma}
\label{lemmatecnico}
Let $C$ be a smooth plane curve of degree $d \geq 5$. Then, $\forall  \ 0 \leq k \leq d-3$ we have 
\begin{enumerate}
\item $R^1p_* (I^k_{\Delta_C} \otimes q^*K_C) \cong {\mathcal O}_C$, 
\item $R^1p_* (I^k_{\Delta_{{\mathbb P}^2}} \otimes q^*(K_C)) \cong {\mathcal O}_{{\mathbb P}^2}$, 
\item $R^1p_* (I^k_{\Delta_{{\mathbb P}^2}} \otimes q^*({\mathcal O}_{{\mathbb P}^2}(d-3))) =0$.\end{enumerate}
\end{lemma}

\begin{proof}

To prove (1), first notice that $\forall x \in C$, we have 
$$H^1(p^{-1}(x), I^k_{\Delta_C} \otimes q^*K_C) \cong H^1(C, K_C(-kx)) \cong H^1(C, K_C) \cong {\mathbb C},$$ for all $k \leq gon(C) -1 = d-2$. 
Then by Grauert's theorem (see e.g. \cite{Ha}, p. 288)  $R^1p_* (I^k_{\Delta_C} \otimes q^*K_C) $ is locally free and its fibre at $x$ is $H^1(C, K_C(-kx))$. For $k= 0$ we have $R^1p_* (q^*K_C) \cong H^1(C, K_C) \otimes {\mathcal O}_C  \cong  {\mathcal O}_C$. 
By induction on $k$, for any $ 0 \leq k \leq gon(C) -1$, the exact sequence \eqref{R} gives a natural map of line bundles
$$f_k: R^1p_* (I^{k+1}_{\Delta_C} \otimes q^*K_C) \rightarrow R^1p_* (I^k_{\Delta_C} \otimes q^*K_C),$$
 that on the fibres at any point $x \in C$ gives the natural isomorphism
$$H^1(C, K_C(-(k+1)x)) \rightarrow H^1(C, K_C(-kx)).$$
Hence the map $f_k$ is an isomorphism, for all $0 \leq k+1 \leq gon(C) -1 = d-2$, so for any $k \leq d-3$ the line bundles $R^1p_* (I^k_{\Delta_C} \otimes q^*K_C)$ are all isomorphic to $R^1p_* (q^*K_C) \cong H^1(C, K_C) \otimes {\mathcal O}_C  \cong  {\mathcal O}_C$.

To prove (2), we argue in the same way, noticing that for all $x \in {\mathbb P}^2$, for all $k \leq gon(C) -1 = d-2$, 
we have 

$$H^1(p^{-1}(x), I^k_{\Delta_{{\mathbb P}^2}} \otimes q^*(K_C)) \cong H^1(C, K_C(-kx)) \cong H^1(C, K_C) \cong {\mathbb C}, \ if \ x \in C,$$ 
$$H^1(p^{-1}(x), I^k_{\Delta_{{\mathbb P}^2}} \otimes q^*(K_C)) \cong H^1(C, K_C)  \cong {\mathbb C}, \ if \ x \not \in C.$$ 
So again by Grauert's Theorem  we know that $R^1p_* (I^k_{\Delta_{{\mathbb P}^2}} \otimes q^*(K_C)) $ is a line bundle and we conclude as in case (1). 

To prove (3), we show that $H^1(p^{-1}(x), I^{k+1}_{\Delta_{{\mathbb P}^2}} \otimes q^*({\mathcal O}_{{\mathbb P}^2}(d-3))) =0$, for all $x \in {\mathbb P}^2$ and for all $k \leq d-3$. We have 
$$H^1(p^{-1}(x), I^{k+1}_{\Delta_{{\mathbb P}^2}} \otimes q^*({\mathcal O}_{{\mathbb P}^2}(d-3))) \cong H^1({\mathbb P}^2, {\mathcal O}_{{\mathbb P}^2}(d-3) \otimes {\mathcal I}^{k+1}_x),$$
for all $x \in {\mathbb P}^2$, where ${\mathcal I}_x$ denotes the ideal sheaf of the point. 

Consider the exact sequence 
$$ 0 \rightarrow {\mathcal I}^{k+1}_x \rightarrow {\mathcal O}_{{\mathbb P}^2} \rightarrow {\mathcal O}_{{\mathbb P}^2} /{\mathcal I}^{k+1}_x \rightarrow 0,$$
tensor it by $L= {\mathcal O}_{{\mathbb P}^2}(d-3)$ and take cohomology. We have the exact sequence 
$$0 \rightarrow H^0({\mathbb P}^2, {\mathcal I}^{k+1}_x \otimes L)  \rightarrow H^0({\mathbb P}^2, L)  \rightarrow H^0({\mathbb P}^2, {\mathcal O}_{{\mathbb P}^2}/{\mathcal I}^{k+1}_x \otimes L) \rightarrow H^1({\mathbb P}^2, {\mathcal I}^{k+1}_x \otimes L))  \rightarrow 0$$
since $H^1({\mathbb P}^2, L) =0$. Then if $L$ is $k$-very ample, the restriction map  $H^0({\mathbb P}^2, L)  \rightarrow H^0({\mathbb P}^2, {\mathcal O}_{{\mathbb P}^2}/{\mathcal I}^{k+1}_x \otimes L)$ is surjective, hence $H^1({\mathbb P}^2, {\mathcal I}^{k+1}_x \otimes L)) =0$.
By a result of Beltrametti and Sommese (\cite{bs}) we know that $L$ is $k$-very ample if $k \leq d-3$. This concludes the proof. 
\end{proof}

Applying the results of Lemma \ref{lemmatecnico} to the exact sequence \eqref{R}, we immediately get the following exact sequences: 

\begin{equation}
\label{RC}
0 \rightarrow R^{k+1}_C \rightarrow R_C^k \rightarrow K_C^{\otimes{(k+1)}} \rightarrow 0, \ \forall k \leq d-3,
\end{equation}

\begin{equation}
\label{G1}
0 \rightarrow G^{k+1} \rightarrow G^k \rightarrow S^{k} \Omega_{{\mathbb P}^2}^1 \otimes K_C \rightarrow 0, \ \forall k \leq d-3,
\end{equation}

\begin{equation}
\label{RX}
0 \rightarrow R^{k+1}_{{\mathbb P}^2} \rightarrow R^{k}_{{\mathbb P}^2} \rightarrow S^{k} \Omega_{{\mathbb P}^2}^1 \otimes  {\mathcal O}_{{\mathbb P}^2}(d-3) \rightarrow 0, \ \forall k \leq d-3.
\end{equation}
Take $k \leq d-3$, then taking restriction to $C$ of the exact sequence \eqref{G1}, we have the following commutative diagram: 

\begin{equation}
\label{diagrammone}
 \xymatrix{
& & &  & 0\ar[d] & \\
& & 0 \ar[d]&0 \ar[d] & S^{k} \Omega_{{\mathbb P}^2}^1(-C) \otimes K_C  \ar[d]& \\
&0\ar[r] & G^{k+1}(-C)\ar[d]\ar[r] &  G^{k}(-C)\ar[d]\ar[r] &  S^{k} \Omega_{{\mathbb P}^2}^1(-C) \otimes K_C\ar[r]\ar[d] & 0\\
&0 \ar[r] &G^{k+1} \ar[r]\ar[d]& G^k \ar[r]\ar[d]& S^{k} \Omega_{{\mathbb P}^2}^1 \otimes K_C\ar[r] \ar[d]& 0\\
0\ar[r]& S^{k} \Omega_{{\mathbb P}^2}^1(-C) \otimes K_C \ar[r]& G^{k+1}_{|C}\ar[r]^{g^{k+1}}\ar[d]&\ar[r] G^{k}_{|C}\ar[r]^{r^k}\ar[d] &S^{k} \Omega_{{\mathbb P}^2}^1 \otimes K_C\ar[r] \ar[d]& 0\\
  &    & 0 &0 &0}
\end{equation}

\begin{proposition}
The image of the map $g^{k+1}$ in diagram \eqref{diagrammone} is $R_C^{k+1}$, for all $k \leq d-3$. So we have the following  exact sequences: 
\begin{equation}
\label{pallino}
0 \rightarrow  S^{k} {\Omega_{{\mathbb P}^2}^1}_{|C}(-3)  \rightarrow G^{k+1}_{|C} \rightarrow R^{k+1}_C \rightarrow 0 
\end{equation}

\begin{equation}
\label{*}
0 \rightarrow   R^{k+1}_C  \rightarrow G^{k}_{|C} \rightarrow S^{k} \Omega_{{\mathbb P}^2}^1 \otimes K_C\rightarrow 0 
\end{equation}

\end{proposition}

\begin{proof}
We have to prove that the image of $g^{k+1}$ is $R_C^{k+1}$. The rest immediately follows splitting the bottom exact sequence of diagram \eqref{diagrammone} in two short exact sequences. 

We do it by induction on $k$. If $k =0$, restricting to $C$ the exact sequence  
$$0 \rightarrow G^1 \rightarrow H^0(C, K_C) \otimes {\mathcal O}_{{\mathbb P}^2} \rightarrow K_C \rightarrow 0,$$ 
we get 
$$0 \rightarrow tor_1^{{\mathcal O}_{{\mathbb P}^2}}(K_C, {\mathcal O}_C) \cong {\mathcal O}_C(-3) \rightarrow G^1_{|C} \stackrel{g^1}\rightarrow H^0(C, K_C) \otimes  {\mathcal O}_C \stackrel{ev}\rightarrow K_C \rightarrow 0. $$

In fact, if we tensor by ${\mathcal O}_C$ the exact sequence 
$$0 \rightarrow L(-C)  \rightarrow L  \rightarrow L_{|C} \cong K_C  \rightarrow 0$$
we get 
$$0 \rightarrow tor_1^{{\mathcal O}_{{\mathbb P}^2}}(K_C, {\mathcal O}_C) \rightarrow L_{|C}(-C) \rightarrow L_{|C} \rightarrow K_C \rightarrow 0,$$
so we obtain $$tor_1^{{\mathcal O}_{{\mathbb P}^2}}(K_C, {\mathcal O}_C) \cong  L_{|C}(-C) \cong {\mathcal O}_C(-3).$$
Now, since $ker(ev) = R_C$ by definition, we get that the image of $g^1$ is $R_C$. 

By induction, assume that  we have the exact sequence 
$$0 \rightarrow  S^{k-1} {\Omega_{{\mathbb P}^2}^1}_{|C}(-3)  \rightarrow G^{k}_{|C} \rightarrow R^{k}_C \rightarrow 0. $$

Take  the cotangent exact sequence 
$$0 \rightarrow {\mathcal O}_C(-d) \rightarrow \Omega^1_{{{\mathbb P}^2}|C} \rightarrow K_C \rightarrow 0,$$
and its $k$-th symmetric power tensored by $K_C$: 
$$0 \rightarrow S^{k-1}\Omega^1_{{{\mathbb P}^2}|C}(-3) \rightarrow S^{k}\Omega^1_{{{\mathbb P}^2}|C} \otimes K_C \rightarrow K_C^{\otimes{(k+1)}}\rightarrow 0.$$

Then we have the following commutative diagram 

\begin{equation}
\label{7}
\xymatrix{
   &  & 0\ar[d] & 0\ar[d] \\
  &&S^{k-1}\Omega^1_{{{\mathbb P}^2}|C}(-3) \ar[r]^{\cong}\ar[d] & S^{k-1}\Omega^1_{{{\mathbb P}^2}|C}(-3)   \ar[d]& \\
   0 \ar[r]  & R^{k+1}_C \ar[d]^{\cong}\ar[r] &  G^k_{|C}\ar[d]\ar[r]^{r^k} &  S^{k}\Omega^1_{{{\mathbb P}^2}|C} \otimes K_C\ar[r]\ar[d] & 0\\
  0\ar[r]&  R^{k+1}_C \ar[r] & R^k_C\ar[r]\ar[d] & K_C^{\otimes{(k+1)}}\ar[r] \ar[d]&0\\
  & &0&0}
\end{equation}

So we see that the image of $g^{k+1}$ which is the kernel of $r^k$  is $R_C^{k+1}$. 
\end{proof}

Now, take diagram \eqref{7}, tensored by $K_C$ and take cohomology. We get the following diagram

\begin{equation}
\label{8}
\xymatrix{
   &  & 0\ar[d] & 0\ar[d] \\
  &&H^0(S^{k-1}\Omega^1_{{{\mathbb P}^2}|C}(d-6)) \ar[r]^{\cong}\ar[d] & H^0(S^{k-1}\Omega^1_{{{\mathbb P}^2}|C}(d-6))   \ar[d]& \\
   0 \ar[r]  & H^0(R^{k+1}_C \otimes K_C) \ar[d]^{\cong}\ar[r] &  H^0(G^k_{|C} \otimes K_C)\ar[d]\ar[r]^{H^0(r^k)} &  H^0(S^{k}\Omega^1_{{{\mathbb P}^2}|C} (2d-6))\ar[r]\ar[d]^{p_2}& \\
  0\ar[r]&  H^0(R^{k+1}_C \otimes K_C) \ar[r] & H^0(R^k_C \otimes K_C)\ar[r]^{\gamma^k_C}\ar[d]^{e^k_C} & H^0(K_C^{\otimes{(k+2)}})\ar[r] \ar[d]&\\
  & &H^1(S^{k-1}\Omega^1_{{{\mathbb P}^2}|C}(d-6)) \ar[r]^{\cong}&H^1(S^{k-1}\Omega^1_{{{\mathbb P}^2}|C}(d-6))}
\end{equation}

\begin{proposition}
\label{main}
For all $k \leq \frac{d-4}{2}$, the map $e^k_C$ is identically zero, hence the image of the $k$-th Gaussian map $\gamma^k_C$ is contained in the image of $p_2$. Thus, for all $k \leq \frac{d-4}{2}$ we have: 
$$corank(\gamma^k_C) \geq corank(p_2) = k(k+3)\frac{2k+3}{2}.$$

\end{proposition}
\begin{proof}
Dualizing equation \eqref{G1}, we obtain 
\begin{equation}
0 \rightarrow (G^{k-1})^*  \rightarrow (G^{k})^*\rightarrow S^{k-1} T_{{{\mathbb P}^2}|C} (3)\rightarrow 0,
\end{equation}
since , $${\mathcal Ext}^1(S^{k-1}\Omega^1_{{{\mathbb P}^2}|C} (d-3), {\mathcal O}_{{\mathbb P}^2}) \cong S^{k-1} T_{{\mathbb P}^2}(3-d) \otimes   {\mathcal Ext}^1({\mathcal O}_C, {\mathcal O}_{{\mathbb P}^2})\cong  $$
$$  \cong S^{k-1} T_{{\mathbb P}^2}(3-d) \otimes {\mathcal O}_C(C) \cong S^{k-1} T_{{{\mathbb P}^2}|C} (3).$$

So we get the following commutative digram
\begin{equation}
\label{9}
\xymatrix{
    & 0\ar[d] & 0\ar[d] \\
 &(G^k)^*(-C) \ar[r]^{\cong}\ar[d] & (G^k)^*(-C)   \ar[d]& \\
   0 \ar[r]  & (G^{k-1})^* \ar[d]\ar[r] &  (G^k)^*\ar[d]\ar[r] &  S^{k-1}T_{{{\mathbb P}^2}|C}(3)\ar[r]\ar[d]^{\cong} & 0\\
  0\ar[r]&  (R^{k}_C)^* \ar[r]\ar[d] & (G^k)^*_{|C}\ar[r]\ar[d] & S^{k-1}T_{{{\mathbb P}^2}|C}(3)\ar[r] &0\\
 & 0& 0&&}.
\end{equation}  

The bottom row of the above diagram is the dual of the exact sequence \eqref{pallino}
$$0 \longrightarrow S^{k-1}\Omega^1_{{{\mathbb P}^2}|C}(-3) \longrightarrow G^k_{|C}\longrightarrow R^k_C \longrightarrow 0,$$
 
since we have the isomorphism 
$$
(G^k_{|C})^* = {\mathcal Hom}_{{\mathcal O}_C }(G^k_{|C},{\mathcal O}_C)  \cong ({\mathcal Hom}_{{\mathcal O}_{{\mathbb P}^2} }(G^k,{\mathcal O}_{{\mathbb P}^2}))_{|C} = (G^k)^*_{|C}.
$$

This last isomorphism can be seen as follows: take the exact sequence 
$$0 \longrightarrow {{\mathcal O}_{{\mathbb P}^2} }(-C)  \longrightarrow {{\mathcal O}_{{\mathbb P}^2} }  \longrightarrow  {\mathcal O}_C  \longrightarrow  0$$
and apply ${\mathcal Hom}_{{\mathcal O}_{\mathbb{P}^2} }(G^k, \cdot)$. We get
$$0 \longrightarrow(G^k)^*(-C)  \longrightarrow (G^k)^*  \longrightarrow  {\mathcal Hom}_{{\mathcal O}_{{\mathbb P}^2} }(G^k,{\mathcal O}_{C})  \longrightarrow  0,$$
since $G^k$ is locally free. 
This shows that ${\mathcal Hom}_{{\mathcal O}_{{\mathbb P}^2} }(G^k,{\mathcal O}_{C}) \cong (G^k)^* _{|C}$. Hence, since we have ${\mathcal Hom}_{{\mathcal O}_{{\mathbb P}^2} }(G^k,{\mathcal O}_{C}) \cong  {\mathcal Hom}_{{\mathcal O}_C }(G^k_{|C},{\mathcal O}_C) $, we are done. 

Considering the two coboundary maps of the cohomology exact sequences of 
the two orizontal exact sequences of diagram \eqref{9}, we get the commutative diagram:

\begin{equation}
\label{10}
\xymatrix{
 & H^0(S^{k-1}T_{{{\mathbb P}^2}|C}(3)) \ar[d]^{\cong}\ar[r]^{\rho} &  H^1((G^{k-1})^*)\ar[d]\\
 & H^0(S^{k-1}T_{{{\mathbb P}^2}|C}(3)) \ar[r]^{(e_C^k)^*} & H^1((R_C^k)^*) &}.
\end{equation}  

Hence  to show that $e_C^k =0$, it suffices to prove that $\rho=0$. 

Equations \eqref{G} and \eqref{RX} for $k=0$  give the commutative diagram
\begin{equation}
\label{k=1}
\xymatrix{
  & &  & & 0\ar[d] \\
 & &0\ar[d]&0 \ar[d] & {\mathcal O}_{{\mathbb P}^2}(-3)   \ar[d]& \\
  & 0 \ar[r]  & R^1_{{\mathbb P}^2} \ar[d]\ar[r] &  H^0(L) \otimes {\mathcal O}_{{\mathbb P}^2} \ar[r]  \ar[d]^{\cong}&  L\ar[r]\ar[d]&0 \\
 & 0\ar[r]&  G^1 \ar[r] \ar[d] & H^0(K_C) \otimes  {\mathcal O}_{{\mathbb P}^2}\ar[r]& K_C\ar[r] \ar[d]&0\\
 & &{\mathcal H}^1:=  {\mathcal O}_{{\mathbb P}^2}(-3)\ar[d]  &&0\\
 &  & 0 & & },
\end{equation}

For all $k \geq 2$ we get the following commutative diagram

\begin{equation}
\label{11}
\xymatrix{
  & &  & & 0\ar[d] \\
 & &0\ar[d]&0 \ar[d] & S^{k-1}\Omega^1_{{\mathbb P}^2}(-3)   \ar[d]& \\
  & 0 \ar[r]  & R^{k}_{{\mathbb P}^2} \ar[d]\ar[r] &  R^{k-1}_{{\mathbb P}^2} \ar[r]  \ar[d]&  S^{k-1}\Omega^1_{{\mathbb P}^2} (d-3)\ar[r]\ar[d]&0 \\
 & 0\ar[r]&  G^k \ar[r] \ar[d] & G^{k-1}\ar[r]\ar[d]& S^{k-1}\Omega^1_{{{\mathbb P}^2}|C}(d-3)\ar[r] \ar[d]&0\\
0 \ar[r]&S^{k-1}\Omega^1_{{\mathbb P}^2}(-3)\ar[r] &{\mathcal H}^k\ar[d]  \ar[r]&{\mathcal H}^{k-1} \ar[d]\ar[r]&0\\
 &  & 0 & 0},
\end{equation}
where we denote by ${\mathcal H}^k$, ${\mathcal H}^{k-1}$ the cokernels of the vertical maps. Notice that they are vector bundles. This can be seen by induction on $k$, since ${\mathcal H}^1=  {\mathcal O}_{{\mathbb P}^2}(-3)$ and using the bottom exact sequence of diagram \eqref{11}. 

Dualising  we get 

\begin{equation}
\label{12}
\xymatrix{
   &  & & 0\ar[d] \\
  && & S^{k-1}T_{{\mathbb P}^2}(3-d)   \ar[d]& \\
   0 \ar[r]  & ({\mathcal H}^{k-1})^*  \ar[d]\ar[r] &  ({\mathcal H}^{k})^* \ar[r]  \ar[d]&  S^{k-1}T_{{\mathbb P}^2} (3)\ar[r]\ar[d]& 0\\
  0\ar[r] &  (G^{k-1})^* \ar[r]  & (G^{k})^*\ar[r]& S^{k-1}T_{{{\mathbb P}^2}|C}(3) \ar[r] \ar[d]& 0\\
  & & & 0}
\end{equation}
Taking the coboundary maps in cohomology we obtain 

\begin{equation}
\label{13}
\xymatrix{
  0\ar[d]& & & \\
     H^0(S^{k-1}T_{{\mathbb P}^2} (3)) \ar[d]^{\cong}\ar[r]^{\alpha} & H^1( ({\mathcal H}^{k-1})^* )\ar[d]& & \\
    H^0(S^{k-1}T_{{{\mathbb P}^2}|C}(3)) \ar[r]^{\rho}  \ar[d]& H^1( (G^{k-1})^*)& & \\
   0& & }
\end{equation}
since $H^1(S^{k-1}T_{{\mathbb P}^2} (3-d)) \cong (H^1(S^{k-1}\Omega^1_{{\mathbb P}^2} (d-6)))^*=0 $, and $H^0(S^{k-1}T_{{\mathbb P}^2} (3-d)) \cong( H^2(S^{k-1}\Omega^1_{{\mathbb P}^2} (d-6)))^*=0$ for $k \leq \frac{d-4}{2}$ (see \cite[Theorem 1.3.1]{et}.

So, to show that $\rho=0$, it suffices to prove that $\alpha =0$. 

We will show that $H^1( ({\mathcal H}^{k-1})^* ) \cong H^1( {\mathcal H}^{k-1}(-3))^* =0$, for all $k  \leq \frac{d-4}{2}$. 

In fact, by definition ${\mathcal H}^1 = {\mathcal O}_{{\mathbb P}^2}(-3)$, hence $$H^1({\mathcal H}^1(-3)) = H^1({\mathcal O}_{{\mathbb P}^2}(-6)) =0.$$
The last row of diagram \eqref{11} tensored by ${\mathcal O}_{{\mathbb P}^2}(-3)$ gives 

$$0 \rightarrow S^{k-2}\Omega^1_{{\mathbb P}^2}(-6) \rightarrow{\mathcal H}^{k -1}(-3)\rightarrow{\mathcal H}^{k-2}(-3) \rightarrow 0.$$
Since $H^1( S^{k-2}\Omega^1_{{\mathbb P}^2}(-6) )=0$,
by induction $H^1( {\mathcal H}^{k-2}(-3)) =0$, then  we get that $H^1( {\mathcal H}^{k-1}(-3)) =0$, forall $k \leq \frac{d-4}{2}$. 
This concludes the proof. 
\end{proof}

So we have  the following 

\begin{theorem}
\label{mainmain}
If $C$ is a smooth plane curve of degree $d \geq 6$,  for all  $0 \leq k \leq \frac{d-6}{2}$, we have 
$$rank(\gamma^k_C) = \frac{2k+3}{2}(d(d-3)-k(k+3)) = \frac{2k+3}{2}(2g-2 -k(k+3)),$$ 
$$corank(\gamma^k_C) = k(k+3)\frac{2k+3}{2}.$$

\end{theorem}
\begin{proof}
The proof immediately follows from Propositions, \eqref{1estimate},  \eqref{main}. 
\end{proof}

\begin{remark}
Notice that the corank of $\gamma^k_C$ is independent of the degree of the plane curve, as it was already shown in  the case $k=1$ (\cite[Remark 4.9]{wahl}). 
\end{remark}

\section{Curves in a product of two curves}
\label{section4}

In this section we study certain higher Gaussian maps $\gamma^k_{X,L}$ for a surface $X=C_1 \times C_2$, where $C_1$ and $C_2$ are smooth projective curves of genus  $g_1$, respectively $g_2$. We prove their surjectivity under suitable assumptions on the degree of the line bundle $L$. Then we use this result to prove the surjectivity of the higher Gaussian maps $\gamma^k_C$ for certain curves $C$ lying on $X$.\\\\
We will use the following result \cite[Theorem 1.7]{el}.
\begin{theorem}
\label{einlaz}
Let $C$ be a smooth curve of genus $g$ and let $L$ and $M$ be line bundles of degree $d$ and $m$ respectively. Let $k \geq 1$ be an integer and assume that $d,m \geq (k+1)(g+1)$.
\begin{itemize}
\item[i] If $d+m \geq (k+1)(2g+2) +2g-1$, then $\gamma^k_{L,M}$ is surjective.
\item[ii] If $C$ is not hyperelliptic $d+m \geq (k+1)(2g+2) +2g-2$,  then $\gamma^k_{L,M}$ is surjective.
\end{itemize}
\end{theorem}
In the following, we denote by $\Delta$ the diagonal in $X \times X$, and by $\Delta_i$, the diagonal in $C_i \times C_i$, $i=1,2$.
\begin{proposition} 
\label{proposition2323}
Let $C_i$, $i=1,2$,  be two smooth curves of genus $g_1$ and $g_2$, respectively. Take  $k \geq 1$ and let $L_1$ be a line bundle on $C_1$ of degree $l_1$ and $L_2$ be a line bundle on $C_2$ of degree $l_2$. Suppose that $g_i$ and $l_i$ satisfy the hypothesis of Theorem \ref{einlaz} with $d=m=l_1$ and $d=m=l_2$ and set $L=p_1^*L_1 \otimes p_2^*L_2$.
 Then the higher Gaussian map $\gamma_{X,L}^k$ is surjective. 
\end{proposition}
\begin{proof}
Recall from section \ref{section2} that  $\gamma^k_{X,L}$ is given by:
$$
\gamma^k_{X,L}: H^0(X \times X, I_{\Delta}^k \otimes p^*L \otimes  q^*L) \rightarrow H^0(X, I^k_{\Delta}\otimes_{\mathcal{O}_{X \times X}} \mathcal{O}_{\Delta}\otimes p^*L \otimes  q^*L).
$$
 Let $\phi_1: (C_1 \times C_2) \times  (C_1\times C_2) \rightarrow (C_1 \times C_1) $ and  $\phi_2: (C_1 \times C_2) \times  (C_1\times C_2) \rightarrow (C_2 \times C_2)$ be the projections. Observe that we have the following:
\begin{equation}
I^k_{\Delta}\otimes_{\mathcal{O}_{X \times X}} \mathcal{O}_{\Delta} \simeq (\phi_1^*I_{\Delta_1} \oplus  \phi_2^*I_{\Delta_2})^{\otimes k}\otimes_{\mathcal{O}_{X \times X}} \mathcal{O}_{\Delta}. 
\end{equation}
From this we obtain the following  commutative diagram:
\begin{center}
\begin{tikzcd}
H^0( \phi_1^{*}I_{\Delta_i} \oplus \phi_2^{*}I_{\Delta_2})^{\otimes k}  \otimes L^{\boxtimes^{2}})  \arrow[r, "\psi"] \arrow[d, ""]
    & H^0(( \phi_1^{*}I_{\Delta_1} \oplus \phi_2^{*}I_{\Delta_2})^{\otimes k} \otimes L^{\boxtimes^{2}} \otimes \mathcal{O}_{\Delta})\
    \arrow[d, "\simeq"] 
    \\
 H^0(I_{\Delta}^k \otimes L^{\boxtimes^{2}})  \arrow[r, black, "\gamma^k_{X,L}" black] 
&H^0(I_{\Delta}^k /I_{\Delta}^{k+1} \otimes L^{\boxtimes^{2}})  \end{tikzcd}
\end{center}
where we have denoted  by $L^{\boxtimes^{2}}=L \boxtimes L$ the tensor product $p^*L \otimes q^*L$. Then it is sufficient to show that $\psi$ is surjective. Now observe that $\psi$ decomposes as a direct sum of maps
\begin{center}
\begin{tikzcd}
 & H^0( (\phi_1^{*}I_{\Delta_i})^{\otimes i}   \otimes (\phi_2^{*}I_{\Delta_2})^{\otimes j}  \otimes L^{\boxtimes^{2}}))\arrow[d,"\psi_{i,j}"]
\\
 &  H^0( (\phi_1^{*}I_{\Delta_i})^{\otimes i}    \otimes (\phi_2^{*}I_{\Delta_2})^{\otimes j} \otimes L^{\boxtimes^{2}} \otimes \mathcal{O}_{\Delta}),
\end{tikzcd}
\end{center}
where $i,j$ vary among the pairs of non negative integers such that $i+j=k$. As in the proof of  \cite[Proposition 2.1.7]{faro}, using K\"unneth formula and the fact $L=p_1^*L_1 \otimes p_1^*L_2$ and $\mathcal{O}_{\Delta} \simeq 
\phi_1^*\mathcal{O}_{\Delta_1} \otimes \phi_2^*\mathcal{O}_{\Delta_2} $, $\psi_{i,j}$ becomes the tensor product $\gamma^i_{1,L_1} \otimes  \gamma^j_{2,L_2}$, where $\gamma^r_{i,L_i}$ is the $r$th Gaussian map on $C_i$ associated with the line bundle $L_i$, $i=1,2$. Since we are assuming that $g_i$ and  $deg(L_i)$ satisfy the hypothesis  of Theorem \ref{einlaz}, each of the Gaussian map is surjective and so each of the $\psi_{i,j}$ is. Hence we conclude that $\psi$ is surjective.
\end{proof}
\begin{remark}
More generally, following a similar approach as the one in Proposition \ref{proposition2323}, one can prove an analogous surjectivity statement for mixed higher Gaussian maps on  $X=C_1 \times C_2$, i.e. higher Gaussian maps associated with line bundles $L=p_1^*L_1 \otimes p_2^*L_2$ and $M=p_1^*M_1 \otimes p_2^*M_2$.
\end{remark}
For future convenience we state the following immediate corollary.
\begin{corollary}
Set $X=C_1 \times C_2$ and let $D_i$, $i=1,2$, be a divisor on $C_i$ of degree $d_i$.  Let $C \in |p_1^*(D_1) \otimes p_2^*(D_2)|$, where $p_i:X:=C_1 \times C_2 \ra C_i$ are the usual projections, and let $H$ be the line bundle $K_X(C)$ on $X$. If 
\begin{equation}
g_i \geq 0, d_i \geq kg_i + k+3, \ i=1,2,
\end{equation}
then the higher Gaussian map $\gamma_{X,H}^k$ is surjective. 
\end{corollary}
\begin{proof}
Apply Proposition \ref{proposition2323} with $L_i=K_{C_i}(D_i)$.
\end{proof}

\begin{theorem}
\label{surjffsrusi}
Let  $k \geq 2$ be an integer. Let $C_i$, $i=1,2$ be smooth projective algebraic curves of genus $g_i$, let $D_i$ be an effective divisor on $C_i$ of degree $d_i$. Denote by $p_i:X:=C_1 \times C_2 \ra C_i$  the projections and suppose that
\begin{itemize}
\item[1.] $g_1 \geq 2, g_2 \geq 1 $ or $g_1 \geq 1, g_2 \geq 2 $,  and $d_i \geq kg_i + k+3$ for $i=1,2$ or,

\item[2.] $g_1=0$, $g_2 \geq 2$, $d_1 > 2(k+1)$, $d_1 > \frac{kd_2}{g_2-1}$, $d_2 \geq kg_2 +k+3 $.

\end{itemize}
 Then for any irreducible smooth curve $C$ in the linear system  $ |p_1^*(D_1) \otimes p_2^*(D_2)|$,  $\gamma^k_{C} $ is surjective. 
\end{theorem}
\begin{proof}
Consider the following commutative diagram
\begin{center}
	\begin{equation}\label{diagrammadue}
		\begin{tikzcd}[sep=tiny]
			H^0(X\times X,\mathcal{I}^k_{\Delta_X}( H\boxtimes H ) )	\arrow{dd}\arrow{r}{\gamma_{X,H}^k}&  H^0(X, S^k  \Omega^1_X \otimes H^{\otimes 2})\arrow{dr}{p_1}&\\
			&  & H^0(C, (S^k  \Omega^1_X \otimes H^{\otimes 2})\restr{C})\arrow{dl}{p_2}\\
			H^0(C\times C,\mathcal{I}^k_{\Delta_C}(K_C\boxtimes K_C) )\arrow{r}{\gamma^k_{C}} &H^0(C,K_C^{\otimes (k+2)})&
		\end{tikzcd}
	\end{equation}
 \end{center}
where  $H$ denotes the line bundle $K_X(C)$, 
$I_{\Delta_{X}}$ and $I_{\Delta_{C}}$  are  ideal sheaves  of the diagonals in $X \times X$ and $C \times C$, respectively. As in section \ref{section2}, the vertical  arrow and $p_1$ are restriction maps and  $p_2$ comes 
from the exact sequence \eqref{p2} tensored by $K_C^{\otimes 2}$.

The surjectivity of $\gamma^k_{C}$ will follow from the surjectivity of $\gamma^k_{X,H}$, $p_1$ and $p_2$. 
\
\\
\
\\
\textbf{Surjectivity of} $\gamma^k_{X,H}$.
\
\\
This follows immediately from Proposition \ref{proposition2323}.
\
\\
\
\\
\textbf{Surjectivity of} $p_1$.
\
\\
In  order to show the surjectivity of $p_1$ it is enough to prove that $H^1(\Sym^k \Omega^1_X \otimes K_X^{\otimes 2}(C))=0$. Denoted by  $K_i$ the canonical bundle of ${C_{i}}$ for $i=1,2$ observe that
\begin{equation}
K_X^{\otimes 2}(C)=p_1^*(K_1^{\otimes 2}) \otimes p_2^*(K_2^{\otimes 2})(C) 
\end{equation}
\begin{equation}
S^k \Omega^1_X \simeq S^k(p_1^*K_1 \oplus p_2^*K_2) \simeq \bigoplus \limits_{\substack{i+j=k \\ i,j \geq 0}} p_1^*(K_1^{\otimes i}) \otimes p_2^*(K_2^{\otimes j}). 
\end{equation}
Hence
\begin{equation}
S^k \Omega^1_X \otimes K_X^{\otimes 2}(C)= \bigoplus \limits_{\substack{i+j=k \\ i,j \geq 0}} (p_1^*(K_1^{\otimes {(i+2)}}(D_1)) \otimes p_2^*(K_2^{\otimes {(j+2)}}(D_2))
\end{equation} 
Then by K\"unneth formula we have that
\begin{align}
&H^1
 (S^k \Omega^1_X \otimes K_X^{\otimes 2}(C))=H^1( \bigoplus \limits_{\substack{i+j = k \\ i,j \geq 0}} (p_1^*(K_1^{\otimes {(i+2)}}(D_1)) \otimes p_2^*(K_2^{\otimes {(j+2)}}(D_2))).  \nonumber
 \\ \nonumber
 &\simeq \bigoplus \limits_{\substack{i+j=k \\ i,j \geq 0}} H^1 (K_1^{\otimes {(i+2)}}(D_1)) \otimes H^0(K_2^{\otimes {(j+2)}}(D_2)) \oplus \bigoplus \limits_{\substack{i+j=k \\ i,j \geq 0}} H^0(K_1^{\otimes {(i+2)}}(D_1)) \otimes H^1(K_2^{\otimes {(j+2)}}(D_2)) \\ \nonumber
 &\simeq \bigoplus \limits_{\substack{i+j=k \\ i,j \geq 0}} H^1 (K_1^{\otimes {(i+2)}}(D_1)) \otimes H^0(K_2^{\otimes {(j+2)}}(D_2)) \oplus \bigoplus \limits_{\substack{i+j=k \\ i,j \geq 0}} H^0(K_1^{\otimes {(i+2)}}(D_1)) \otimes H^1(K_2^{\otimes {(j+2)}}(D_2)). \\ \nonumber
\end{align}
Now observe that since by Serre duality
\begin{equation*}
H^1 (K_1^{\otimes{(i+2)}}(D_1)) \simeq H^0(K_1^{\otimes {(-i-1)}}(-D_1))  \ \text{and} \ H^1 (K_2^{\otimes {(j+2)}}(D_2)) \simeq H^0(K_2^{\otimes {(-j-1)}}(-D_2)),
\end{equation*}
to have the desired vanishing it is sufficient that for any $0  \leq i,j \leq k$ 
\begin{align}
 d_1 > -(i+1)(2g_1-2); \nonumber
 \\ \nonumber
  d_2 > -(j+1)(2g_2-2).
 \end{align}
 If $g_1, g_2 >0$,  these of course occur for every $d_1,d_2 \geq 1$. If $g_1=0$ $(g_2=0)$ the conditions become
\begin{equation}
d_1 > 2(k+1), \ \ \ d_2 > 2(k+1). 
\end{equation}
\textbf{Surjectivity of} $p_2$
\
\\
In order to show to the surjectivity of $p_2$ by \eqref{p2} it is enough to prove that $H^1(S^{k-1} \Omega^1_{X_{|_{C}}}  \otimes K_X^{\otimes 2}(C) )=0$. Observe that 
\begin{equation}
S^{k-1}\Omega^1_{X_{|_{C}}}  \otimes K_C^{\otimes 2}(-C)=S^{k-1}((p_1^*K_{C_1} \oplus p_2^*K_{C_2})_{|_{C}})  \otimes K_C^{\otimes 2}(-C),
\end{equation}
which is isomorphic to 
\begin{equation}
\bigoplus \limits_{\substack{i+j=k-1 \\ i,j \geq 0}} p_1^*K_{C_1}^{\otimes{(i+2)}}(D_1)_{|_{C}} \otimes p_2^*K_{C_2}^{\otimes{(j+2)}}(D_2)_{|_{C}}.
\end{equation}
Then we get
$$
H^1(S^{k-1}\Omega^1_{X_{|_{C}}}  \otimes K_C^{\otimes 2}(-C))=\bigoplus \limits_{\substack{i+j=k-1 \\ i,j \geq 0}} H^1(p_1^*K_{C_1}^{\otimes{(i+2)}}(D_1)_{|_{C}} \otimes p_2^*K_{C_2}^{\otimes {(j+2)}}(D_2)_{|_{C}}).
$$
Again, by Serre duality, in order to have the vanishing of $H^1$, it is enough that 
$$
2g(C)-2 < \text{deg} (p_1^*K_{C_1}^{\otimes{(i+2)}}(D_1)_{|_{C}} \otimes p_2^*K_{C_2}^{\otimes{(j+2)}}(D_2)_{|_{C}}),
$$ 
 i.e.
$$
(2g_1-2)d_2+(2g_2-2)d_1+2d_1d_2 < (i+2)(2g_1-2)d_2+(j+2)(2g_2-2)d_1+2d_1d_2
$$
for every $i,j$. If  $g_1 \geq 2$ and $ g_2 \geq 1$ or viceversa then the conditions hold for every $d_1, d_2 \geq 1$. If $g_1=0$, $g_2 \geq 1$  the conditions become
\begin{equation}
-2d_2+(2g_2-2)d_1< -(i+2)2d_2+(j+2)(2g_2-2)d_1
\end{equation}
for every $i, j \geq 0 , i+j=k-1$, which is equivalent to the condition relative to $i=k-1$, $j=0$. That is
\begin{equation*}
(g_2-1)d_1>d_2k.
\end{equation*}
Then if $g_1=0$ we need $g_2 \geq 2$ and $d_1>\frac{kd_2}{(g_2-1)}$. Analogously if $g_2=0$ we need $g_1 \geq 2$ and $d_2 > \frac{kd_1}{(g_1-1)}$.  Comparing all the conditions we conclude.
\end{proof}
\begin{remark}
\label{remarkdsdsre4}
If the divisors $D_1$ and $D_2$ in  Proposition \ref{surjffsrusi} are general divisors on $C_1$ and $C_2$, then the linear system $|p_1^*(D_1) \otimes p_2^*(D_2)|$  is base-point-free, and  there actually exists a smooth irreducible curve $C \in |p_1^*(D_1) \otimes p_2^*(D_2)|$. This follows from the fact that under the hypothesis of Proposition \ref{surjffsrusi}, we always  have $d_i \geq g_i+1$, and we use classical Brill-Noether theory. The general curve $C \in |p_1^*(D_1) \otimes p_2^*(D_2)|$ is a smooth curve of genus
$$
g(C)=1+(g_2-1)d_1+(g_1-1)d_2+d_1d_2.
$$
In particular the lowest genus (depending on $k$)  is obtained by choosing  $g_1=0$ and $g_2=2$, $d_1=3k+3$ and $d_2=2k+3$. This is 
$$
6k^2+17k+13.
$$
Notice that  this is  higher than the bound in \cite[Theorem D]{ro}.
\end{remark}
%
A consequence of the previous Proposition is the following.
\begin{corollary}
\label{surhjcewahalgaussi}
Let $k \geq 2$. For all $g_i$ and $d_i$ satisfying the hypothesis of  Proposition \ref{surjffsrusi} the general curve of genus  
\begin{equation}
g=1+(g_2-1)d_1+(g_1-1)d_2+d_1d_2,
\end{equation}
has surjective $k$th Gaussian-Wahl map.
\end{corollary}

\section{Higher Gaussian maps for curves on Enriques surfaces}
\label{section5}
In this section we show the surjectivity of the higher Gaussian maps $\gamma^k_C$ when $C$ is a curve lying on a sufficiently positive linear system of a unnodal Enriques surface. We recall that an Enriques surface is called unnodal if it does not contain any $-2$ curves. The general Enriques surface is unnodal. 

Before going into the proof we recall the definition of the $\phi$-function. In the following $X$ will always be an Enrique surface.
\begin{definition}
Let $H$ be a line bundle on $X$ with $H^2 >0$. Then 
\begin{equation}
\phi(H):=\{H \cdot F: F \in \operatorname{Pic}(X), \  F^2=0, \ F \not\equiv 0 \},
\end{equation}
\end{definition}
where $\operatorname{Pic}(X)$ is the Picard group of $X$ and $â€˜â€˜\equiv"$ stands for numerically equivalent. Recall that  if $X$ is an Enriques surface, then $K_X^{\otimes 2} \simeq \mathcal{O}_X$. In particular $\phi(H)=\phi(H \otimes K_X)$ for any $H.$
\begin{theorem}
\label{enriques}
Let $X$ be an unnodal Enriques surface and $k \in \mathbb{N}$. If $C$ is a smooth curve with $\phi(\mathcal{O}_X(C)) > 4(k+2)  $, then $\gamma^k_C$ is surjective. If $k=1$ it is sufficient to require. $\phi(\mathcal{O}_X(C)) > 6 $
\end{theorem}

\begin{proof}
The proof  again considers  the following commutative diagram.
\begin{center}
	\begin{equation}
		\begin{tikzcd}[sep=tiny]
			H^0(X\times X,\mathcal{I}^k_{\Delta_X}( H\boxtimes H ) )	\arrow{dd}\arrow{r}{\gamma_{X,H}^k}&  H^0(X, S^k  \Omega^1_X \otimes H^{\otimes 2})\arrow{dr}{p_1}&\\
			&  & H^0(C, (S^k  \Omega^1_X \otimes H^{\otimes 2})\restr{C})\arrow{dl}{p_2}\\
			H^0(C\times C,\mathcal{I}^k_{\Delta_C}(K_C\boxtimes K_C) )\arrow{r}{\gamma^k_{C}} &H^0(C,K_C^{\otimes (k+2)})&
		\end{tikzcd}
	\end{equation}
 \end{center}
where  $H$ denotes the line bundle $K_X \otimes \mathcal{O}_X(C)$. As usual we show the surjectivity of $\gamma^k_{X,H}$, $p_1$ and $p_2$.

Consider the Hilbert scheme of two points $\operatorname{X}^{[2]}$ of  $X$. We recall that $\operatorname{Pic}(\operatorname{X}^{[2]}) \simeq \operatorname{Pic}(X) \oplus \mathbb{Z}B$, where $2B$ is the exceptional divisor of the Hilbert-Chow morphism $\operatorname{X}^{[2]} \rightarrow X^{(2)}$, and $X^{(2)}$ is the second symmetric product. In the following, if $L \in \operatorname{Pic}(X)$, we denote by $\tilde{L}$ the corresponding line bundle on $\operatorname{X}^{[2]}$.

By \cite[Theorem A]{ro}, for the surjectivity of $\gamma^k_{X,H}$,  it is sufficient to show that $H^1(\operatorname{X}^{[2]},\tilde{H}-(k+2)B)=0$. This immediately follows from Kodaira vanishing since by \cite[Corollary  4.2]{farospelta}
$\tilde{H}-(k+2)B-K_{\operatorname{X}^{[2]}}=\widetilde{H-K_X}-(k+2)B$ is ample.   

As in the previous sections, in  order to show the surjectivity of $p_1$,  it is enough to prove that $H^1(S^k \Omega^1_X \otimes K_X^{\otimes 2}(C))=H^1(S^k \Omega^1_X(C))=0$. This follows immediately from \cite[Lemma 4.4]{farospelta}.

Finally, recall that $p_2$ comes for the $k$th symmetric power of the conormal exact sequence twisted bt $H^{\otimes 2}$:
\begin{equation*}
    0 \rightarrow S^{k-1} \Omega^1_{X_{\mid C}}(C) \rightarrow  S^k \Omega^1_{X_{\mid C}}(2C)\rightarrow  K_C^{\otimes{(k+2)}} \rightarrow  0.
\end{equation*}
Hence, as usual, to show the surjectivity of $p_2$ is enough to prove that $H^1(X,S^{k-1} \Omega^1_{X_{\mid C}}(C))=0$. For the proof of this vanishing see \cite[4.7]{farospelta} and  \cite[Remark 4.5]{farospelta} for the case $k=1$.

\end{proof}

\end{document}